\documentclass[letterpaper, 10 pt, conference]{ieeeconf}

\IEEEoverridecommandlockouts
\usepackage{url}
\usepackage{amsmath}
\usepackage{amssymb}
\usepackage{color}

\newcommand{\rg}{\rangle}
\newcommand{\define}{\widehat{=}}
\newcommand{\xx}{\mathbf{x}}
\newcommand{\uu}{\mathbf{u}}
\newcommand{\fb}{\mathbf{f}}

\newcommand{\ud}{\mathrm{d}}

\newtheorem{theorem}{Theorem}

\newtheorem{example}[theorem]{Example}

\newtheorem{definition}[theorem]{Definition}

\newenvironment{remark}{\textbf{Remark}}

\usepackage[noend,ruled]{algorithm2e}
\SetInd{2.7ex}{1.5ex}
\LinesNumbered
\SetKwBlock{Block}{begin}{end} \SetKwFor{While}{while}{do}{}
\SetKwFor{ForAll}{for all}{do}{} \SetKwFor{For}{for}{do}{}
\SetKwIF{If}{ElseIf}{Else}{if}{then}{else if}{else}{}

\SetAlgoSkip{medskip}

\title{\bf Automatically Discovering Relaxed Lyapunov
Functions\\ for Polynomial Dynamical Systems
\thanks{This work is
supported in part by the projects NSFC-90718041, NSFC-60736017 and
NSFC-60970031.}}

\author{\authorblockN{Jiang Liu,  Naijun Zhan and Hengjun Zhao}
\authorblockA{State Key Lab. of Computer Science, Institute of
Software, Chinese Academy of Sciences\\
No.~4 South Fourth Street,
Zhong Guan Cun, Beijing, 100190, P.R. China\\
Email: \{liuj, znj, zhaohj\}@ios.ac.cn}}

\begin{document}

\maketitle

\begin{abstract}
The notion of Lyapunov function plays a key role in design and
verification of dynamical systems, as well as hybrid and
cyber-physical systems. In this paper, to analyze the asymptotic
stability of a dynamical system, we generalize standard Lyapunov
functions to \emph{relaxed Lyapunov functions} (RLFs), by
considering higher order Lie derivatives of certain functions along
the system's vector field. Furthermore, we present a complete method
to automatically discovering polynomial RLFs for polynomial
dynamical systems (PDSs). Our method is complete in the sense that
it is able to discover all polynomial RLFs by enumerating all
polynomial templates for any PDS.
\end{abstract}

\section{Introduction}
The notion of Lyapunov function plays a very important role in
design and verification of dynamical systems, in particular, in
performance analysis, stability analysis and controller synthesis of
complex dynamical and controlled systems
\cite{Haddad-Che,Khalil,Li-Slotine}. \,In recent years, people
realized that the notion is quite helpful to safety verification of
hybrid and cyber-physical systems as well \cite{Tabuada}.

However, the following two issues hinder the application of Lyapunov
functions in practice.  Firstly,  it is actually not necessary for
the first-order Lie derivative of a Lyapunov function to be strictly
negative  to guarantee  asymptotic stability, which is shown by
\emph{LaSalle's Invariance Principle} \cite{Khalil}. Such a
condition  could limit to scale up the method. Secondly, in general
there is no effective way so far to find Lyapunov functions,
although many methods have been proposed by different experts using
their field expertise.

To address the above two issues, in this paper, we first generalize
the standard concept of Lyapunov function to \emph{relaxed Lyapunov
function} (RLF) for asymptotic stability analysis. Compared with the
conventional definition of Lyapunov function,  the first non-zero
higher order Lie derivative of RLF is required to be negative,
rather than its first-order Lie derivative. Such a relaxation
extends the set of admissible functions that can be used to prove
asymptotic stability.

Another contribution of this paper is that we present a complete
method to automatically discovering polynomial RLFs for polynomial dynamical systems (PDSs).
The basic idea of  our method  is to predefine a parametric polynomial as a template of
RLF first, and then utilize
the Lie derivatives of the template at different orders to
generate constraint on the parameters, and
finally solve the resulting constraint.
Our method is complete in the sense that it
is able to generate all polynomial RLFs by enumerating all polynomial templates for any PDS.
\newline

\noindent {\bf{Related Work.}} In \cite{RS}, the same terminology
``relaxed Lyapunov function" is used, with a different definition.

The idea of applying higher order Lie derivatives to analyze
asymptotic stability is not new. For example,  in \cite{Butz,Mei-Ni}
the authors resorted to certain linear combinations  of higher order
Lie derivatives with non-negative coefficients such that the
combination is always negative. This method could be included in the
framework of vector Lyapunov functions method
\cite{Matrosov,Ner-Had}.  Our method is essentially different from
theirs because an RLF only requires its first non-zero higher order
Lie derivative to be negative.

In the literature, there is a lot of work on  constructing Lyapunov
functions.  For instance, in
\cite{Gurvits,Liberzon,Agrachev-Liberzon}, methods for constructing
common quadratic Lyapunov functions for linear systems were
proposed, which were generalized in  \cite{SNS} and
\cite{Vu-Liberzon} for nonlinear systems wherein the  generated Lyapunov functions are not necessarily quadratic. Another useful technique is the linear matrix inequality (LMI) method introduced in \cite{Jo-Ran} and \cite{Pet-Ben}, which enables us to utilize the results of numerical optimization for discovering  piecewise quadratic Lyapunov functions. Based on sums-of-squares (SOS)
decomposition and semi-definite programmng (SDP) \cite{Parrilo}, a
method for constructing piecewise high-degree polynomial and
piecewise non-polynomial Lyapunov functions was proposed in
\cite{Pra-Pap} and \cite{Pra-Pap1}.  The SOS and SDP based method
was also used in \cite{TP} to search for control Lyapunov functions
for polynomial systems. In \cite{SXX}, the authors proposed a new
method for computing Lypunov functions for polynomial systems by
solving semi-algebraic constraint using their tool DISCOVERER
\cite{Xia07}. Approaches to constructing Lyapunov functions beyond
polynomials using radial basis functions were proposed in
\cite{Giesl2,Giesl3}.

Our method has the following features compared to the related work.
Firstly, it generates relaxed Lyapunov functions rather than
conventional Lyapunov functions. Secondly, it is able to discover
all polynomial RLFs by enumerating all polynomial templates
for any PDS, whereas the Krasovskii's method \cite{Kra} and Zubov's
method \cite{Zubov} can only produce  Lyapunov functions of special
forms. Thirdly, the LMI method and SOS method are numerical, while
our method is symbolic, which means it could provide a
mathematically rigorous framework for the stability analysis of
polynomial dynamical systems.

\noindent {\bf{Structure:}} The rest of this paper is organized as
follows. In Section \ref{sec:foundations}, the
theoretical foundations are presented. 
Section \ref{sec:RLF} shows a new criterion for asymptotic stability using the notion of relaxed Lyapunov functions.
In Section \ref{sec:main-result} we present a sound and complete
method and a corresponding algorithm for automatically  discovering
polynomial RLFs on polynomial dynamical systems. The method is
illustrated by an example in Section \ref{sec:example}. Finally, we
conclude this paper and discuss possible future work in Section
\ref{sec:conclusion}.

\section{Theoretical Foundations}\label{sec:foundations}

In this section, we present the fundamental materials based on which
we develop our method.

\subsection{Polynomial Ideal Theory}

Let $\mathbb{K}$ be an algebraic field, and $\mathbb{K}[x_1, x_2,
\dots, x_n]$ denote the polynomial ring over $\mathbb{K}$.
Customarily, let $\xx$ denote the $n$-tuple $(x_1, \cdots, x_n)$.
Then  $\mathbb{K}[x_1, x_2, \dots, x_n]$ can be written as
$\mathbb{K}[\xx]$ for short, and a polynomial in $\mathbb{K}[x_1,
x_2, \dots, x_n]$ can simply be written as $p(\xx)$ or $p$.
Particularly, $\mathbb{K}$ will be taken as the real field $\mathbb
R$ in this paper, and $\xx$ takes value from the $n$-dimensional
Euclidean space $\mathbb{R}^n$.

In our method we will use polynomials with undetermined
coefficients, called parametric polynomials or \emph{templates}.
Such polynomials are denoted by $p(\uu,\xx)$, where
$\uu=(u_1,u_2,\ldots,u_t)$ is a $t$-tuple of parameters. A
parametric polynomial $p(\uu,\xx)$ in $\mathbb{R}[x_1, x_2, \dots,
x_n]$ with real parameters can be seen equivalently as a regular
polynomial in $\mathbb{R}[u_1, u_2, \ldots, u_t, x_1, x_2, \ldots,
x_n]$. Given $\uu_0\in \mathbb{R}^t$, we call the polynomial
$p_{\uu_0}(\xx)$ resulted by substituting $\uu_0$ for $\uu$ in
$p(\uu,\xx)$ an \emph{instantiation }of $p(\uu,\xx)$.

The following are some fundamental results relative to
polynomial ideals, which can be found in \cite{clo}.
\begin{definition}
A subset $I\subseteq \mathbb{K}[\xx]$ is called an ideal iff
\begin{enumerate}
\item[(a)] $0\in I$;
\item[(b)] If $p(\xx), g(\xx)\in I$, then $p(\xx)+g(\xx)\in I$;
\item[(c)] If $p(\xx)\in I$, then $p(\xx)h(\xx)\in I$  for any $h(\xx)\in
\mathbb{K}[\xx]$.
\end{enumerate}
\end{definition}

It is easy to check that if $p_1, \cdots, p_m\in \mathbb{K}[\xx]$, then
\[\<p_1, \cdots, p_m\rg=\{\sum_{i=1}^{m}p_ih_i\mid \forall i\in [1,m].\, h_i \in
\mathbb{K}[\xx]\}\] is an ideal. In general, we say
an ideal $I$ is {\itshape generated} by polynomials $g_{1}, g_2,
\dots, g_k\in \mathbb{K}[\xx]$ if $I=\<g_{1}, g_2,
\dots, g_k\rg,$ where all $g_i$ for $i \in [1,k]$ are called {\itshape
generators} of $I$.  In fact, we have

\begin{theorem}[Hilbert Basis Theorem]\label{basis}\ \
Every ideal $I$ $\subseteq \mathbb{K}[\xx]$ has a
finite generating set. That is, $I=\<g_{1}, g_2, \dots, g_k\rg$ for
some $g_{1}, g_2, \dots, g_k \in \mathbb{K}[\xx]$.
\end{theorem}

From this result, it is easy to see that

\begin{theorem}[Ascending Chain Condition]\label{ACC}\ \
For any ascending chain
\[I_1\subseteq I_2 \subseteq \cdots \subseteq I_m \subseteq \cdots\] of
ideals in polynomial ring $\mathbb{K}[\xx]$, there
must be an $N$ such that for all $m\geq N$, $I_m=I_N$.
\end{theorem}

\subsection{Dynamical Systems and Stability}
We summarize some fundamental theories of dynamical systems here.
For details please refer to  \cite{Haddad-Che,Khalil,Li-Slotine}.
\subsubsection{Dynamical Systems}
We consider autonomous dynamical systems modeled by first-order ordinary differential equations
\begin{equation}\label{eq:ode}
  \dot \xx= \fb(\xx) \enspace,
\end{equation}
where $\xx\in\mathbb{R}^n$ and $\fb$ is a vector function from $\mathbb R^n$ to $\mathbb R^n$, which is also called a vector field in $\mathbb R^n$.

In this paper, we focus on special nonlinear dynamical systems whose vector fields are defined by polynomials.
\begin{definition}[Polynomial Dynamical System]\label{dfn:ads}\ \
Suppose $\fb=(f_1,f_2,\cdots,f_n)$ in (\ref{eq:ode}). Then (\ref{eq:ode}) is called a \emph{polynomial dynamical system} (PDS for short) if for every $1\leq i\leq n$, $f_i$ is a polynomial in $\mathbb R[\xx]$.
\end{definition}

If $\fb$ satisfies the local Lipschitz condition, then given
$\xx_0\in \mathbb R^n$, there exists a unique solution $\xx(t)$ of
(\ref{eq:ode}) defined on $(a, b)$ with $a< 0< b$ s.t.
$$\forall t\in (a,b).\,{\ud \xx(t)\over \ud t} = \fb (\xx(t))\quad \mathrm{and}\quad  \xx(0)=\xx_0\enspace .$$
We call $\xx(t)$ on $[0,b)$ the \emph{trajectory}  of (\ref{eq:ode}) starting from initial point $\xx_0$.

Let $\sigma(\xx)$ be a function from $\mathbb R^n$ to $\mathbb R$. Suppose both $\sigma$ and $\fb$ are differentiable in $\xx$ at any
order $n\in \mathbb N$. Then we can inductively define the {\itshape
Lie derivatives} of $\sigma$ along $\fb$, i.e.
 $L^k_{\fb}\sigma: \mathbb R^n\rightarrow \mathbb R$ for $k\in \mathbb N$, as follows:
\begin{itemize}
\item $L^0_{\fb} \sigma(\xx)=\sigma(\xx)$,
\item $L^k_{\fb} \sigma(\xx)=\left( \frac{\partial}{\partial \xx} L^{k-1}_{\fb} \sigma(\xx), \fb(\xx)\right)$, for $k>0$,
\end{itemize}
where $(\cdot, \cdot)$ is the inner product of two vectors, i.e. $(
\,(a_1, \ldots, a_n), ( b_1, \ldots, b_n )\,) =\sum_{i=1}^n a_ib_i.$

Polynomial functions are sufficiently smooth, so given a PDS
$\mathcal{P}$ and a polynomial $p$, the vector field $\fb$ of
$\mathcal{P}$ satisfies the local Lipschitz condition, and the
higher order Lie derivatives of $p$ along $\fb$ are well defined and
are all polynomials.  For a parameterized polynomial $p(\uu,\xx)$,
we can define $L_{\fb}^k p(\uu,\xx): \mathbb R^n\rightarrow \mathbb
R$ by seeing  $\uu$ as undetermined constants rather than variables.
In the sequel we will implicitly employ these facts.
\begin{example}\label{eg:Lie-derv}
Suppose $\fb=(-x,y)$ and $p(x,y)=x+y^2$. Then $L^0_{\fb} p=x+y^2$, $L^1_{\fb} p=-x+2y^2$, $L^2_{\fb} p=x+4y^2$, $L^3_{\fb} p=-x+8y^2$ .
\end{example}

\subsubsection{Stability}
The following are classic results of stability of dynamical systems in the sense of Lyapunov.
\begin{definition}\
A point $\xx_e\in \mathbb R^n$ is called an {\itshape equilibrium or
critical point} of (\ref{eq:ode}) if  $\fb(\xx_e)=\mathbf{0}$.
\end{definition}

We assume $\xx_e=\mathbf{0}$ w.l.o.g from now on.

\begin{definition}\
Suppose $\mathbf{0}$ is an equilibrium of (\ref{eq:ode}). Then
\begin{itemize}
\item
$\mathbf{0}$ is called Lyapunov stable if for any
$\epsilon>0$, there exists a $\delta>0$ such that if $\|
\xx_0\|<\delta$, then for the corresponding solution $\xx(t)$ of (\ref{eq:ode}), $\|\xx(t)\|<\epsilon$ for all $t\geq 0$.

\item
$\mathbf{0}$ is called asymptotically stable if it is
Lyapunov stable and there exists a $\delta>0$ such that for any $\|
\xx_0\|<\delta$, the
corresponding solution $\xx(t)$ of (\ref{eq:ode}) can be extended to infinity and $\lim_{t\rightarrow
\infty}\xx(t)=\mathbf{0}$.
\end{itemize}
\end{definition}

Lyapunov  first provided a sufficient condition, using so-called \emph{Lyapunov
function}, for the Lyapunov stability as follows.
\begin{theorem}[Lyapunov Stability Theorem]\label{Lya}\
Suppose $ \mathbf{0}$ is an equilibrium point of
(\ref{eq:ode}). If there is an open set $ U\subset \mathbb R^n$ with
 $\mathbf{0} \in U$ and a continuous differentiable function $V: U\rightarrow \mathbb R$
such that
\begin{enumerate}
\item[(a)] $V(\mathbf{0})=0$,
\item[(b)]$V(\xx)>0$ for all $\xx\in U\backslash \{\mathbf{0}\}$  and
\item[(c)]$L^1_{\fb} V(\xx)\leq 0$ for all $\xx\in U$,
\end{enumerate}
then $\mathbf{0}$ is a stable equilibrium of (\ref{eq:ode}).  Moreover, if
condition (c) is replaced by
\begin{enumerate}
\item[(c')] $L^1_{\fb} V(\xx)< 0$ for all $\xx\in U\backslash\{\mathbf{0}\}$,
\end{enumerate}
then $\mathbf{0}$ is an asymptotically stable equilibrium of (\ref{eq:ode}).
Such  $V$ is called a {\itshape Lyapunov function}.
\end{theorem}

For asymptotic stability,  we have
Barbashin-Krasovskii-LaSalle (BKLS) Principle which relaxes condition $(c\textrm{'})$ in Theorem \ref{Lya}.

\begin{theorem}[BKLS Principle] \label{bkls} \
Suppose there exists $V$ satisfying the conditions (a), (b) and (c)
in Theorem \ref{Lya}. If the set $\mathcal M\,\define \, \{\xx\in \mathbb R^n\mid
L^1_{\fb} V(\xx)=0\}\cap U$ does not contain any trajectory of the
system besides the trivial trajectory $\xx(t)\equiv \mathbf{0}$,
then $\mathbf{0}$ is  asymptotically stable.
\end{theorem}

Inspired by Theorem \ref{bkls}, we will define \emph{relaxed
Lyapunov function} (RLF for short) in the subsequent section, which
guarantees the asymptotic stability of an equilibrium of a dynamical
system.

\section{Relaxed Lyapunov Function}\label{sec:RLF}
Intuitively, a Lyapunov function requires that any trajectory
starting from $\xx_0\in U$ cannot leave the region $\{\xx\in \mathbb
R^n\mid V(\xx)\leq V(\xx_0)\}$. While, in the asymptotic stability
case, the corresponding $V$ forces any trajectory starting from
$\xx_0\in U$ to transect the boundary $\{\xx\in \mathbb R^n\mid
V(\xx)= V(\xx_0)\}$ towards the set $\{\xx\in \mathbb R^n\mid
V(\xx)< V(\xx_0)\}$. It is clear that   $L_{\fb}^1 V(\xx)<0$ is only
a sufficient condition to guarantee  asymptotic stability.
  When a point $\xx$ satisfies $L_{\fb}^1 V(\xx)=0$,
the transection requirement may still be met if the first non-zero
higher order Lie derivative of $V$ at $\xx$ is negative. To
formalize this idea, we give the following definition.

\begin{definition}[Pointwise Rank]\label{dfn:point-rank} \
Let $\mathbb{N}^+$ be the set of positive natural numbers.
Given sufficiently smooth function $\sigma$ and vector filed $\fb$, the {\itshape pointwise  rank} of $\sigma$ w.r.t. $\fb$ is defined as the function $\gamma_{\sigma,\fb}: \mathbb R^n \rightarrow \mathbb N\cup\{\infty\}$ given by
\begin{equation*}
  \gamma_{\sigma,\fb}(\xx)= \left\{
  \begin{array}{l}
    \infty , \qquad \mbox{if } \forall k\in \mathbb{N}^+.\, L^k_{\fb}\sigma(\xx)= 0, \\
    \min\{k\in \mathbb{N}^+\mid L^k_{\fb}\sigma(\xx)\neq 0\} , \quad \mbox{otherwise.}
  \end{array} \right.
\end{equation*}
\end{definition}
\begin{example}\label{eg:point-rank}
For $\fb=(-x,y)$ and $p(x,y)=x+y^2$, by Example \ref{eg:Lie-derv}, we have $\gamma_{p,\fb}(0,0)=\infty$, $\gamma_{p,\fb}(1,1)=1$, $\gamma_{p,\fb}(2,1)=2$.
\end{example}

\begin{definition}[Transverse Set]\label{dfn:transet}\
Given sufficiently smooth function $\sigma$ and vector field $\fb$, the
\emph{transverse set} of $\sigma$ w.r.t $\fb$ is defined as
\begin{equation*}
  \mathrm{Trans}_{\sigma,\fb}\,\define \, \{\xx\in \mathbb R^n \mid  \gamma_{\sigma,\fb}(\xx)<\infty \wedge  L^{\gamma_{\sigma,\fb}(\xx)}_{\fb} \sigma(\xx)<0\} \enspace .
\end{equation*}
\end{definition}

Intuitively, $\mathrm{Trans}_{\sigma,\fb}$ consists of those points
at which the first non-zero high order Lie derivative of $\sigma$
along $\fb$ is negative. Now we can relax condition $(c\textrm{'})$
in Theorem \ref{Lya} and get a stronger result for asymptotic
stability.

\begin{theorem}\label{LyaSta}\
Suppose $ \mathbf{0} $ is an equilibrium point of
(\ref{eq:ode}). If there is an open set $ U\subset \mathbb R^n$ with
 $\mathbf{0}\in U$ and a sufficiently smooth function $V: U\rightarrow \mathbb R$ s.t.
\begin{enumerate}
\item[(a)] $V(\mathbf{0})=0$,

\item[(b)] $V(\xx)>0$ for all $\xx\in U \backslash \{\mathbf{0}\}$  and

\item[(c)] $\xx \in \mathrm{Trans}_{V,\fb}$  for all $\xx\in U \backslash \{\mathbf{0}\}$,
\end{enumerate}
then $\mathbf{0}$ is an asymptotically stable equilibrium
point of (\ref{eq:ode}).
\end{theorem}
\begin{proof}
First notice that condition (c) implies
$L^1_{\fb} V(\xx)\leq 0$ for all $\xx\in U\backslash\{\mathbf{0}\}$.
In order to show the asymptotic stability of $ \mathbf{0}$, according to Theorem \ref{bkls},
it is sufficient to show that $\mathcal M\,\define \, \{\xx \in \mathbb R^n\mid
L^1_{\fb} V(\xx)=0\}\cap U$ contains no
nontrivial trajectory of the dynamical system.

If not, let
$\xx(t)$, $t\geq 0$ be such a trajectory
contained in $\mathcal M$ other than  $\xx(t)\equiv  \mathbf{0}$. Then $L^1_{\fb}V(\xx(t))=0$ for all $t\geq 0$.  Noting that $\xx_0=\xx(0)\in \mathrm{Trans}_{V,\fb}$, we can get
the Taylor expansion of $L^1_{\fb}V(\xx(t))$ at $t=0$:
\begin{align*}
L^1_{\fb}V(\xx(t)) &= L_{\fb}^1 V(\xx_0) + L_{\fb}^2 V(\xx_0)\cdot t +
L_{\fb}^3 V(\xx_0)\cdot \frac{t^2}{2!} +\cdots \\
&= L_{\fb}^{\gamma_{V,\fb}(\xx_0)} V(\xx_0)\cdot
\frac{t^{\scriptscriptstyle{\gamma_{V,\fb}(\xx_0)}-1}}{(\gamma_{V,\fb}(\xx_0)-1)!}+\cdots\enspace .
\end{align*}
By Definition \ref{dfn:transet}, there exists an $\epsilon>0$ s.t. $\forall t \in (0,\epsilon). \,L^1_{\fb}p(\xx(t))<0,$
which contradicts the assumption.
\end{proof}

\begin{definition}[Relaxed Lyapunov Function] \
We refer to the function $V$ in Theorem \ref{LyaSta}
as a {\itshape relaxed Lyapunov function}, denoted by
RLF.
\end{definition}

In the next section, we will explore how to discover polynomial RLFs automatically for PDSs.

\section{Automatically Discovering RLFs for PDSs}\label{sec:main-result}
Given a PDS, the process of automatically  discovering polynomial RLFs is  as follows:
\begin{itemize}
  \item a template, i.e. a parametric polynomial $p(\uu,\xx)$, is predefined as a potential   RLF;
  \item the conditions for $p(\uu,\xx)$ to be an RLF are translated into an equivalent formula $\Phi$
  of
   the decidable \emph{first-order theory of reals} \cite{tarski51};
  \item   constraint $\Phi '$ on parameters $\uu$, or equivalently a set $S_{\uu}$ of all $t$-tuples subject to $\Phi'$, is obtained by applying  \emph{quantifier elimination} (QE for short. See \cite{Redlog,qepcad})  to $\Phi$, and any instantiation of $\uu$ by $\uu_0 \in S_{\uu}$ yields an RLF $p_{\uu_0}(\xx)$.
\end{itemize}

\subsection{Computation of Transverse Set}\label{sec:auxiliary}
Correct translation of the three conditions in Theorem \ref{LyaSta} is crucial to our method. In particular, we have to show that for any polynomial $p(\xx)$ and polynomial vector field $\fb$, the transverse set $\mathrm{Trans}_{p,\fb}$ can be represented by first order polynomial formulas. To this end, we first give several theorems by exploring the properties of Lie derivatives and polynomial ideas.

In what follows, given a parameterized polynomial $p(\uu,\xx)$, all Lie derivatives $L_{\fb}^k p$ are seen as polynomials in $\mathbb{R}[\uu,\xx]$. Besides, we will use the convention that $\bigvee_{i\in \emptyset}\eta_i =\textit{false}$ and $\bigwedge_{i\in \emptyset}\eta_i =\textit{true}$, where $\eta_i$ is logical formula.
\begin{theorem}[Fixed Point Theorem]\label{GIT}\
Given $p\,\define\,p(\uu,\xx)$, if $L^{i}_{\fb}p\in \<L^{1}_{\fb}p,$  $ \cdots, L^{i-1}_{\fb}p\rg$,
then for all $m>i$,
$L^{m}_{\fb}p\in \<L^{1}_{\fb}p, \cdots,
L^{i-1}_{\fb}p\rg.$
\end{theorem}
\begin{proof}
We prove this fact by induction. Assume
$L^{k}_{\fb}p\in \<L^{1}_{\fb}p,$  $ \cdots, L^{i-1}_{\fb}p\rg$ for
$k\geq i$.  Then there are $g_j\in \mathbb
R[\uu, \xx]$ s.t.
$L^{k}_{\fb}p=\sum_{j=1}^{i-1} g_jL^{j}_{\fb}p$.
By the definition of Lie derivative it follows that
\begin{align}
L^{k+1}_{\fb} p&~=~(\frac{\partial}{\partial \xx}L^k_{\fb}p, \fb) \nonumber\\
~&~=~(\frac{\partial}{\partial \xx}\sum_{j=1}^{i-1} g_jL^{j}_{\fb}p, \fb) \nonumber\\
~&~=~(\sum_{j=1}^{i-1}L^{j}_{\fb} p \frac{\partial}{\partial \xx}g_j+\sum_{j=1}^{i-1}g_j
 \frac{\partial}{\partial \xx}L^{j}_{\fb}p, \fb) \nonumber\\
~&~=~\sum_{j=1}^{i-1}(\frac{\partial}{\partial \xx}g_j, \fb)L^{j}_{\fb} p+\sum_{j=1}^{i-1}g_jL^{j+1}_{\fb}p\nonumber\\
~&~=~\sum_{j=1}^{i-1}(\frac{\partial}{\partial \xx}g_j, \fb)L^{j}_{\fb}
p+\sum_{j= 2}^{i-1}g_{j-1}L^{j}_{\fb}p+g_{i-1}L^i_{\fb}p. \nonumber
\end{align}
By induction hypothesis, $L^i_{\fb}p\in \<L^{1}_{\fb}p,
\cdots, L^{i-1}_{\fb}p\rg$, so $L^{k+1}_{\fb}p\in\<L^{1}_{\fb}p,
\cdots, L^{i-1}_{\fb}p\rg$. By induction, the fact follows immediately.
\end{proof}

\begin{theorem}\label{thm:upper}\
Given $p\,\define\,p(\uu,\xx)$, the number
$$N_{p,\fb}=\min\{i\in\mathbb{N}^+\mid L^{i+1}_{\fb}p\in \<L^{1}_{\fb}p, \cdots,
L^{i}_{\fb}p\rg\}$$
is well defined and computable.
\end{theorem}
\begin{proof}
First it is easy to show that $N_{p,\fb}$ has an equivalent expression
$\label{eq:npf}N_{p,\fb}=\min\{i\in\mathbb{N}^+\mid I_{i+1}=I_i\}$,
where $I_i=\<L^{1}_{\fb}p, \cdots,
L^{i}_{\fb}p\rg\subseteq \mathbb R[\uu,\xx]$.
Notice that
$$I_1\subseteq\\I_2\subseteq\cdots\subseteq I_k\cdots $$
forms an ascending chain of ideals. By Theorem \ref{ACC},
$N_{p,\fb}$ is well-defined. Computation of $N_{p,\fb}$ is actually an \emph{ideal membership} problem, which can be solved by computation of \emph{Gr\"{o}bner basis} \cite{clo}.
\end{proof}
\begin{example}
For $\fb=(-x,y)$ and $p(x,y)=x+y^2$, by Example \ref{eg:Lie-derv}, we have $L^2_{\fb}p\notin \<L^1_{\fb}p\rg$ and $L^3_{\fb}p\in \<L^1_{\fb}p,L^2_{\fb}p\rg$, so $N_{p,\fb}=2$\,.
\end{example}

\begin{theorem}[Rank Theorem]\label{Para-Dir-Rank} \ \
Suppose that $p\,\define\, p(\uu,\xx)$. Then  for all $\xx\in \mathbb
R^n$ and all $\uu_0\in \mathbb{R}^t$, $\gamma_{p_{\uu_0},\fb}(\xx)<\infty$ {implies} $\gamma_{p_{\uu_0},\fb}(\xx)\leq N_{p,\fb}.$
\end{theorem}
\begin{proof}
If the conclusion is not true, then there exist $\xx_0\in\mathbb R^n$ and $\uu_0\in\mathbb R^t$ s.t.
$$N_{p,\fb}<\gamma_{p_{\uu_0},\fb}(\xx_0)<\infty \enspace .$$
By Definition \ref{dfn:point-rank}, $\xx_0$ satisfies
\begin{equation*}
L_{\fb}^1 p_{\uu_0}=0\wedge \cdots\wedge L_{\fb}^{N_{p,\fb}}p_{\uu_0}=0
\wedge L_{\fb}^{\gamma_{p_{\uu_0},\fb}(\xx_0)}p_{\uu_0}\neq 0\enspace .
\end{equation*}
Then by Theorem \ref{thm:upper} and \ref{GIT},
for all $m>N_{p,\fb}$, we have $L_{\fb}^m p_{\uu_0}(\xx_0)=0$. In particular, $L_{\fb}^{\gamma_{p_{\uu_0},\fb}(\xx_0)}p_{\uu_0}(\xx_0)=0$, which contradicts $L_{\fb}^{\gamma_{p_{\uu_0},\fb}(\xx_0)}p_{\uu_0}(\xx_0)\neq 0$.
\end{proof}

Now we are able to show the computability of $\mathrm{Trans}_{p,\fb}$.
\begin{theorem}\label{thm:comput-trans}\
Given a parameterized polynomial $p\,\define\,p(\uu,\xx)$ and polynomial vector field $\fb$, for any $\uu_0\in \mathbb{R}^t$ and any $\xx\in\mathbb{R}^n$, $\xx\in \mathrm{Trans}_{p_{\uu_0},\fb}$ if and only if $\uu_0$ and $\xx$ satisfy $\varphi_{p,\fb}$, where
\begin{equation}\label{eq:varphi}
\varphi_{p,\fb}\,\define \,
\bigvee_{i=1}^{N_{p,\fb}}\varphi_{p,\fb}^{i}\,,\quad \mbox{and}
\end{equation}
\begin{equation}\label{eq:varphi-i}
\varphi_{p,\fb}^{i}\,\define \,  (\bigwedge_{j=1}^{i-1}L_{\fb}^j
p(\uu,\xx)=0)\wedge L_{\fb}^i p(\uu,\xx)<0\,.
\end{equation}
\end{theorem}
\begin{proof}
($\Rightarrow$) Suppose $\xx\in \mathrm{Trans}_{p_{\uu_0},\fb}$. By Definition \ref{dfn:transet} $\xx$ satisfies
\begin{equation}\label{eq:trans2}
L_{\fb}^1 p_{\uu_0}=0\wedge \cdots\wedge L_{\fb}^{\scriptscriptstyle{\gamma_{p_{\uu_0},\fb}(\xx)-1}}p_{\uu_0}=0
\wedge L_{\fb}^{\scriptscriptstyle{\gamma_{p_{\uu_0},\fb}(\xx)}}p_{\uu_0}<0\enspace .
\end{equation}
By Theorem \ref{Para-Dir-Rank}, $\gamma_{p_{\uu_0},\fb}(\xx)\leq N_{p,\fb}$. Then it is easy to check that (\ref{eq:trans2}) implies (\ref{eq:varphi}) when $\uu=\uu_0$.

($\Rightarrow$) If $\uu_0$ and $\xx$ satisfy $\varphi_{p,\fb}$, then from Definition \ref{dfn:transet} we can see that $\xx\in \mathrm{Trans}_{p_{\uu_0},\fb}$ holds trivially.
\end{proof}

\subsection{A Sound and Complete Method for Generating RLFs}
Based on the results established in Section \ref{sec:auxiliary}, we
can give a sound and complete method for automatically generating
polynomial RLFs on PDSs.

Given $\xx =(x_1,x_2,\ldots, x_n)\in \mathbb R^n$, let $\|\xx \|=
\sqrt{\sum_{i=1}^n x_i^2}$ denote the Euclidean norm of $\xx$. Let
$\mathcal B(\xx,d)=\{\mathbf{y}\in \mathbb R^n\mid
\|\mathbf{y}-\xx\|< d\}$ for any $d> 0$. Then  our main result can
be stated as follows.
\begin{theorem}[Main Result]\label{thm:main}
Given a PDS $\dot{\xx}=\fb(\xx)$  with $\fb(\mathbf{0})=\mathbf{0}$
and a parametric polynomial $p\,\define\,p(\uu,\xx)$. Let  $r_0\in
\mathbb R$ and $\uu_0=(u_{1_0},u_{2_0},\ldots,u_{t_0})\in \mathbb
R^t$. Then $p_{\uu_0}$ is an RLF in $\mathcal {B}(\mathbf{0},r_0)$
if and only if
\[(u_{1_0},u_{2_0},\ldots, u_{t_0},r_0) \in
\mathrm{QE}( \phi_{p,\fb})\enspace ,\] where
\begin{equation}\label{eqn:phi}
\phi_{p,\fb}\,\define \, \phi_{p,\fb}^{1}\wedge
\phi_{p,\fb}^{2} \wedge \phi_{p,\fb}^{3} \enspace ,
\end{equation}
\begin{equation}\label{eq:phi-1}
\phi_{p,\fb}^{1}\,\define \, p(\uu,\mathbf{0})=0\enspace ,
\end{equation}
\begin{equation}\label{eq:phi-2}
\phi_{p,\fb}^{2}\,\define \, \forall \xx.(\| \xx\|^2>0 \, \wedge
\,\| \xx\|^2 < r^2 \rightarrow p(\uu,\xx)>0)\,,
\end{equation}
\begin{equation}\label{eq:phi-3}
\phi_{p,\fb}^{3}\,\define \,\forall \xx.( \| \xx\|^2>0 \, \wedge
\,\| \xx\|^2 < r^2 \rightarrow \varphi_{p,\fb})\enspace .
\end{equation}
\end{theorem}
\begin{proof}
First, in Theorem \ref{LyaSta}, the existence of an open set $U$ is
equivalent to the existence of an open set $\mathcal
{B}(\mathbf{0},r_0)$. Then according to Theorem
\ref{thm:comput-trans}, it is easy to check that (\ref{eq:phi-1}),
(\ref{eq:phi-2})  and (\ref{eq:phi-3}) are direct translations of
conditions $(a)$, $(b)$ and $(c)$ in Theorem \ref{LyaSta}.
\end{proof}

According to Theorem \ref{thm:main}, we can follow the three steps at the beginning of Section \ref{sec:main-result} to  discover polynomial RLFs on PDSs. This method is ``complete" because we can discover all possible polynomial RLFs by enumerating all polynomial templates.

\subsection{Implementation}
To construct $\phi_{p,\fb}$ in Theorem \ref{thm:main}, we need to compute $N_{p,\fb}$ in advance,  which is time-consuming. What is worse, when $N_{p,\fb}$ is a large number the resulting $\phi_{p,\fb}$ can be a huge formula, for which QE is difficult.
For analysis of asymptotic stability,
one RLF is enough. Therefore if an RLF can be obtained by solving constraint involving merely lower order Lie
derivatives, there's no need to resort to higher order ones. Regarding this, we give an incomplete but more efficient implementation of Theorem \ref{thm:main}, by constructing $\phi_{p,\fb}$ and searching for RLFs in a stepwise manner.

Let
$$
\psi_{p,\fb}^{i}\,\define \, \bigwedge_{j=1}^{i-1} L_{\fb}^j
p(\uu,\xx)=0,\, \mbox{for}\, i\geq 1\enspace ,
$$
\begin{eqnarray*}
\theta_{p,\fb}^{i}\,\define \, \forall \xx.(\|
\xx\|^2>0 \,
\wedge \,\| \xx\|^2 < r^2 \wedge \psi_{p,\fb}^{i}
\rightarrow L_{\fb}^i p(\uu,\xx)<0)
\end{eqnarray*}
and
\begin{eqnarray*}
\bar\theta_{p,\fb}^{i}\,\define \, \forall \xx.(\|
\xx\|^2>0 \,
\wedge \,\| \xx\|^2 < r^2 \wedge \psi_{p,\fb}^{i}
\rightarrow L_{\fb}^i p(\uu,\xx)\leq0)\,.
\end{eqnarray*}

Intuitively, for $\xx$ satisfying $\psi_{p,\fb}^{i}$, we have to
impose constraints $\theta_{p,\fb}^{i}$ or
$\bar\theta_{p,\fb}^{i}$ on the $i$-th higher order Lie derivative
of $p$ along $\fb$.
Now the  RLF generation algorithm (RLFG for short) can be formally stated as
follows.
\begin{algorithm}[ht]\label{alg:lfg}
\caption{Relaxed Lyapunov Function Generation}
Input: $\fb \in \mathbb{R}[x_1,\ldots,x_n]^n$
with $\fb(\mathbf{0})=\mathbf{0}$, \\$p\in \mathbb{R}[u_1,\ldots,u_t, x_1,\ldots,x_n] $\\
Output: $Res\subseteq \mathbb{R}^{t+1}$ \\
$i:=1$; $temp:=\emptyset$; $L_{\fb}^1 p:=(\frac{\partial}{\partial \xx} p,\fb)$;\\
$Res^{0}:=\mathrm{QE}(\phi_{p,\fb}^{1} \wedge \phi_{p,\fb}^{2} )$;\\
\If{$Res^{0}=\emptyset$}{return $\emptyset$;}\Else{
\textbf{repeat}\\
\qquad $temp:= Res^{i-1}\cap \mathrm{QE}(\theta_{p,\fb}^{i})$;\\
\qquad \If{$temp \neq \emptyset\,$}{\qquad return $temp\,$;}\qquad
\Else{\qquad $Res^{i}:=Res^{i-1}\cap \mathrm{QE}
(\bar\theta_{p,\fb}^{i})$;\\
\qquad\If{$Res^{i}=\emptyset$}{\qquad return
$\emptyset$;} \qquad \Else{\qquad $i:=i+1$;\\
\qquad  $L_{\fb}^i p:=(\frac{\partial}{\partial \xx} L_{\fb}^{i-1}
p,\fb)$;}}\textbf{until}\,\,\,$L_{\fb}^i
p\in \langle L_{\fb}^1 p, L_{\fb}^2 p, \ldots,
L_{\fb}^{i-1}p\rangle$;}
return $\emptyset$;
\end{algorithm}

\begin{remark} Formula $\phi_{p,\fb}^{1}$ and $\phi_{p,\fb}^{2}$ in
line 5 are defined in (\ref{eq:phi-1}) and (\ref{eq:phi-2}); QE in line 5, 10
and 14 is done in a computer algebra tool like  Redlog
\cite{Redlog} or QEPCAD \cite{qepcad}; in line 20
the loop test can  be done by calling the \emph{IdealMembership} command in Maple$^{\scriptscriptstyle{\textrm{TM}}}$ \cite{Maple}.
\end{remark}

The idea of Algorithm \ref{alg:lfg} is: at the $i$-th step, we
search for an RLF using constraint constructed from Lie derivatives
with order no larger than $i$. If this fails to produce a solution,
then we add the $(i+1)$-th  order Lie derivative to the constraint.
This process continues until either we succeed in finding a
solution,  or we can conclude that there is no RLF with the
predefined template, or we get to the $N_{p,\fb}$-th iteration,
which means no solution exists at all.

Correctness of the
algorithm  RLFG is guaranteed by the following theorem.

\begin{theorem}\label{thm:crrct}
For Algorithm \ref{alg:lfg}, we have
\begin{description}
\item[1) Termination.]\qquad\qquad RLFG terminates for any valid input.
\item[2) Soundness.] \qquad\qquad If $(\uu,r) = \, (u_1,u_2,\ldots,u_t,r)
\in Res$, then $p_{\uu}(\xx)$ is an RLF in
$\mathcal{B}(\mathbf{0},r)$.
\item[3) Weak Completeness.]\qquad\qquad\qquad\qquad If $Res=\emptyset$ then there does
not exist an RLF in the form of $p(\uu,\xx)$.
\end{description}
\end{theorem}
\begin{proof}
\begin{itemize}
\item[(1)] The loop condition is $L_{\fb}^i p \notin \langle L_{\fb}^1 p, L_{\fb}^2 p, \ldots,$ $L_{\fb}^{i-1}p\rangle$. By Theorem \ref{thm:upper}, RLFG can run at most $N_{p,\fb}$
many iterations.
\item[(2)] Suppose $Res^{0},Res^{1},\ldots, Res^{k}$ is the longest sequence generated by RLFG when it terminates. We can inductively prove that this sequence satisfies the following properties.
\begin{itemize}
\item[(P1)] $0\leq k \leq N_{p,\fb}$.
\item[(P2)] $Res^{i}=\mathrm{QE}(\phi_{p,\fb}^{1}\wedge \phi_{p,\fb}^{2} \wedge \bar\phi_{p,\fb}^{i})$, \,\, for $0\leq i \leq k$, where \begin{eqnarray*}\bar\phi_{p,\fb}^{i}\,\define \,\forall \xx. \Big{(}\| \xx\|^2>0 \,
 \wedge \,\| \xx\|^2 < r^2)\longrightarrow
  \\
 \big( (\bigvee_{j=1}^i \varphi_{p,\fb}^{j}\,)
\,\vee \psi_{p,\fb}^{i+1}\big{)}\Big{)}\enspace.
\end{eqnarray*}
\item[(P3)] $\mathrm{QE}(\phi_{p,\fb}^{1}\wedge \phi_{p,\fb}^{2} \wedge \tilde\phi_{p,\fb}^{i})=\emptyset,$ \, for $1\leq i \leq k$, where
$$\tilde\phi_{p,\fb}^{i}\,\define \, \forall \xx. \Big( (\| \xx\|^2>0 \,
 \wedge \,\| \xx\|^2 < r^2)
\rightarrow \bigvee_{j=1}^i \varphi_{p,\fb}^{j}\Big)\, .$$
\item[(P4)] $Res=\emptyset $ if and only if either $Res^{k}=\emptyset$ or $k=N_{p,\fb}$; otherwise $Res=\mathrm{QE}( \phi_{p,\fb}^{1}\wedge \phi_{p,\fb}^{2}\wedge \tilde\phi_{p,\fb}^{k+1})$.
\end{itemize}
Suppose $(\uu,r)\in Res$, then by (P1), (P4) and
(\ref{eqn:phi}) we can get $Res \subseteq
\mathrm{QE}({\phi_{p,\fb}})$. Thus $(\uu,r) \in \mathrm{QE}({\phi_{p,\fb}}) $ and $p_{\uu}(\xx)$ is an RLF according to Theorem \ref{thm:main}.
\item[(3)] Suppose $Res=\emptyset$, then by (P4) we have either $k=N_{p,\fb}$ or $Res^{k}=\emptyset$. If $k=N_{p,\fb}(\geq 1)$,
then by (P3) and  (\ref{eqn:phi}) we get
$\mathrm{QE}(\phi_{p,\fb})=\emptyset$; if $Res^{k}=\emptyset$, from  (P1), (P2), (\ref{eqn:phi}) as well as  the validity of
\[
\Big(\bigvee_{j=1}^{N_{p,\fb}} \varphi_{p,\fb}^{j}\Big)
\rightarrow \Big( \big{(}\bigvee_{j=1}^k \varphi_{p,\fb}^{j}\,\big{)}
\,\vee \psi_{p,\fb}^{k+1}\Big), \,\, 0\leq k\leq N_{p, \fb}, \]
 we have
$\mathrm{QE}(\phi_{p,\fb})\subseteq Res^{k}=\emptyset$.  So far we have proved $Res = \emptyset$ implies  $\mathrm{QE}(\phi_{p,\fb})= \emptyset$. Again by applying Theorem \ref{thm:main} we get the final conclusion.
\end{itemize}
\end{proof}

\section{Example}\label{sec:example}

We illustrate our method for RLF generation using the following example.
\newpage

\begin{example}\label{eg:smpl}
Consider the nonlinear dynamical system \begin{equation}\label{eqn:eg1}
\left(\begin{array}{c} \dot x\\ \dot y
\end{array}\right)=\left(\begin{array}{c} -x+y^2\\ -xy
\end{array}\right)
\end{equation}
with a unique equilibrium  point $O(0,0)$. We want to establish the asymptotic stability of $O$.

First,
the linearization of (\ref{eqn:eg1}) at $O$ has the coefficient matrix
\begin{displaymath}
A=\left(\begin{array}{cc} -1 & 0\\ 0 & 0 \end{array}\right)
\end{displaymath}
with eigenvalues $-1$ and $0$, so none of the principles of stability for linear systems apply. Besides, a homogeneous quadratic Lyapunov function $x^2+axy+by^2$ for verifying asymptotic stability of (\ref{eqn:eg1}) does not exist in $\mathbb R^2$, because
$$
\mathrm{QE}\Big{(}\forall x\forall y. \left(
\begin{array}{c}
(x^2+y^2> 0\rightarrow x^2+axy+by^2>0)\\
\wedge \,(2x\dot x + ay\dot x +a x\dot y+ 2by\dot y<0)
\end{array}
\right)\Big{)}
$$
is $\textit{false}$.
However, if we try to find an RLF in $\mathbb R^2$ using
the simple template $p\,\define \, x^2+ay^2$, then Algorithm \ref{alg:lfg} returns $a=1$ at the third iteration. This means (\ref{eqn:eg1}) has an RLF $x^2+y^2$, and $O$ is asymptotically stable.

\end{example}

From this example, we can see that RLFs really extend the class of functions that can be used for asymptotic stability analysis, and our method for automatically discovering RLFs can save us a lot of effort in finding conventional Lyapunov functions in some cases.

\section{Conclusions and Future Work}\label{sec:conclusion}

In this paper, we first generalize the notion of Lyapunov functions
to \emph{relaxed Lyapunov functions} by considering the higher order Lie derivatives of a smooth function along a vector field. The main advantage of RLF is that it provides us more probability of certifying asymptotic stability. We also
propose a method for automatically discovering polynomial
RLFs for polynomial dynamical systems. Our method is complete in the sense that we can enumerate all potential polynomial RLFs by enumerating all polynomial templates for a given PDS.  We
believe that our methodology could serve as a mathematically
rigorous framework for the asymptotic stability analysis.

The main disadvantage of our approach is the high computational complexity: the complexity of the first-order quantifier elimination over the closed fields  of reals is doubly exponential \cite{dh88}.
Currently we are considering improving the efficiency QE on first order polynomial formulas in special forms, and it will be the main focus of our future work.

\section{Acknowledgments}
The authors thank Professors C. Zhou, L. Yang,
B. Xia and W. Yu
for their helpful discussions on the topic.


\end{document}